\documentclass{article}

\usepackage{amsmath,amsthm,amsfonts,amscd,amssymb, tikz-cd} 
\usepackage{verbatim}      	
\usepackage{epsfig}         	
\usepackage{url}		
\usepackage{epic,eepic,latexsym}
\usepackage{graphicx}
\usepackage{color}
\usepackage{lcd}
\usepackage{mathrsfs}
\usepackage[centering,text={15.5cm,22cm},
        marginparwidth=20mm,dvips]{geometry}
\usepackage[linktocpage]{hyperref}
\usepackage{lscape}
\usepackage{tabularx}
\usepackage{marginnote}
\usepackage[all]{xy}
\usepackage{enumitem}
\usepackage{stmaryrd}
\usepackage{tipa}
\usepackage{float}

\newtheorem{lemma}{Lemma}[section]
\newtheorem{theorem}[lemma]{Theorem}
\newtheorem{proposition}[lemma]{Proposition}

\newtheorem{conjecture}[lemma]{Conjecture}

\newtheorem{claim*}{Claim}
\newtheorem{remark}[lemma]{Remark}
\newtheorem{definition}[lemma]{Definition}

\newtheorem{cal*}{Calculation}

\title{The Model Orbit in $G_2$}
\author{Man-Wai Cheung}
\date{}

\begin{document}

\maketitle
\abstract{In this article, we decompose the ring of regular functions on the nilpotent orbit of dimension 8 for the complex $G_2$ in which every irreducible representation of $G_2$ appears exactly once. This confirms the predication of McGovern and we have shown that his proposed representation attaching to this orbit is unitarizable.}

\section{Introduction}

One of the main objectives in representation theory of reductive Lie groups is to understand their unitary duals. Various systematic constructions of irreducible unitary representations have been developed but it is still far from complete for most of the real reductive groups. Recently, attaching irreducible unitary representations to nilpotent coadjoint orbits has been a testing ground for finding new irreducible unitary representations. Unfortunately, there is a lack of a canonical construction of such association. However, reasonable guesses have been proposed according to the orbit method predictions. Among all of them, Vogan \cite{E8} gave the following conjecture concerning the attachment of irreducible unitary representations to the real nilpotent orbits with admissible data:

\begin{conjecture} \label{conj:1.1}
Let $G$ be a real reductive Lie group and $(\lambda_\mathbb{R},
\chi_\mathbb{R})$ be a nilpotent $\mathbb{R}$-admissible orbit
datum. Suppose that $(\lambda_{\theta}, \chi_{\theta})$ is the
corresponding $\theta$-admissible data and
$\mathcal{V}_{\chi_\theta}$ is the corresponding homogeneous
vector bundle over the $K_\mathbb{C}$-orbit $K_\mathbb{C} \cdot
\lambda_\theta$. If the boundary of the orbit closure $\overline{G
\cdot \lambda_\mathbb{R}}$ has codimension at least four, then
there exists an irreducible unitary representation
$\pi(\lambda_\mathbb{R}, \chi_\mathbb{R})$ such that the space of
$K$-finite vectors of $\pi(\lambda_\mathbb{R}, \chi_\mathbb{R})$
is isomorphic to the space of algebraic sections of
$\mathcal{V}_{\chi_\theta}$.
\end{conjecture}

The geometric motivation of this association comes from the quantization theory originated in relating the classical and quantum theories of certain mechanical systems. Traditionally, the mathematical machineries are framed on that of symplectic manifolds with Hamiltonian $G$-actions, from which such spaces can be realized as $G$-equivariant coverings of nilpotent orbits via the moment map (in fact, a $G$-equivariant Poisson structure is enough). Such theory works well in elliptic orbits, but it fails in general real nilpotent orbits due to the absence of a $G$-equivariant complex polarization. The new insight coming from the Sekiguchi correspondence enables the tranformation of a real nilpotent coadjoint orbit of $G$ to a complex nilpotent $K_\mathbb{C}$-orbit on $\mathfrak{p}$ $K$-equivariantly; where by a natural $K$-equivariant complex structure on the orbit gives a quantization of the $K$-action of the orbit. This explains why Vogan imposes conditions on the $K$-finite part of the quantization, and the essence of Conjecture \ref{conj:1.1} is that the quantization of $K$-action can be extended to the whole $G$.

The main purpose of this paper is to verify Conjecture \ref{conj:1.1} for the case $G$ being the complex simple Lie group of type $G_2$ and $G \cdot \lambda_\mathbb{R}$ being the ``model orbit'' of complex dimension $8$. In this case, we conduct a series of simplification; namely, first, since $G$ and $G \cdot \lambda_\mathbb{R}$ are complex, $K_\mathbb{C}$ can be naturally identified  with $G$ as complex groups and $G \cdot \lambda_\mathbb{R}$ and $K_\mathbb{C} \cdot \lambda_\theta$ as complex homogeneous spaces; second, the codimension assumption is automatic for complex groups. Thus Conjecture \ref{conj:1.1} is reduced into the following form.

\begin{conjecture} \label{conj:1.2}
Suppose $G$ is a complex reductive Lie group, and $X= G \cdot \lambda = G/ G(\lambda)$ is a nilpotent coadjoint orbit. Suppose we are given an irreducible $G$-equivariant local system on $X$; equivalently, an irreducible representation $\chi$, $V_{\chi}$ of the finite group $G(\lambda) / G(\lambda_0)$, or an indecomposable $G$-equivariant holomorphic vector bundle $\mathcal{V}_{\chi}$ on $X$ with flat connection. Then there is attached to $\chi$ an irreducible unitary representation $\pi (\lambda, \chi)$ of $G$. The space of $K$-finite vectors of $\pi ( \lambda, \chi)$ is isomorphic to the space of algebraic sections of the bundle $\mathcal{V}_{\chi}$. 
\end{conjecture}
Then we will show in proposition \ref{prop:2} that it is the model orbit in the sense that every irreducible representations of $K$ appears with multiplicity one in its ring of regular functions. The advantage is that McGovern \cite{McG} has a proposed candidate for the attached unitary representation, namely the irreducible spherical representation with parameter one fourth the sum of the positive roots. We will show in Section \ref{sec:last} that the unitarity of this representation follows from the results of Levasseur-Smith \cite{LS} and Huang \cite{Huang}.


\section{Nilpotent orbits}

In this section, we will set up our notations and state the main theorem. We will also indicate briefly how to reduce Conjecture \ref{conj:1.1} to our main theorem.

Let $G$ be a real reductive Lie group in Harish-Chandra's class and $\mathfrak{g}_0$ be the real Lie algebra of $G$. Choose a maximal compact subgroup $K$ of $G$, or equivalently we have fixed a choice of Cartan involution $\theta$ of $G$. Then on Lie algebra level, we obtain the corresponding Cartan decomposition $\mathfrak{g}_0 = \mathfrak{k}_0 + \mathfrak{p}_0$ as the eigenspace decomposition of $\theta$. Let $K_\mathbb{C}$ be the complexification of $K$ and we adopt to the convention that linear spaces with subscript $0$ denote real vector spaces and removing the subscript denote their complexifications. A coadjoint orbit is by definition an orbit in $i\mathfrak{g}_0^*$ under the coadjoint action of $G$. The $\mathbb{R}$-nilpotent cone in $i\mathfrak{g}_0^*$ is defined to be $$\mathcal{N}^*_\mathbb{R} = \{ \lambda \in  i\mathfrak{g}_0^* |\, \mathbb{R}^+ \lambda \subset G \cdot \lambda\}.$$
The coadjoint orbits which lie inside $\mathcal{N}^*_\mathbb{R}$ are said to be nilpotent. On the other hand, there is a natural action of $K_\mathbb{C}$ on $\left(\mathfrak{g}/\mathfrak{k}\right)^*$, we can define the corresponding $\theta$-nilpotent cone to be $$\mathcal{N}^*_\theta =\{\lambda \in \left(\mathfrak{g}/\mathfrak{k}\right)^* | \,\mathbb{C}^* \lambda \subset K_\mathbb{C} \cdot \lambda\}.$$
We call an $K_\mathbb{C}$-orbit in $\left(\mathfrak{g}/\mathfrak{k}\right)^*$ $\theta$-nilpotent if it is contained in $\mathcal{N}^*_\theta$.

\subsection{Admissibile datum}
The first issue is to explain the admissibility conditions mentioned in Conjecture \ref{conj:1.1}. The notion of admissibility is introduced by Duflo (or at least for the formulation that we are adopting to). Roughly speaking, his idea of introducing admissibility condition is to replace the weaker condition, called the integral condition, in which orbit method applied nicely to those integral coadjoint orbits of compact Lie groups, for more general type of topological groups. His approach is to make use of the symplectic structure of the coadjoint orbit. Given $\lambda_\mathbb{R} \in \mathcal{N}^*_\mathbb{R}$, the orbit $\mathcal{O} = G\cdot \lambda_\mathbb{R}$ equips with natural symplectic form $\omega$. More precisely, if $G(\lambda_\mathbb{R})$ denotes the isotropy subgroup of $\lambda_\mathbb{R}$ with Lie algebra $\mathfrak{g}(\lambda_\mathbb{R})_0$, the orbit is naturally identified with $G/G(\lambda_\mathbb{R})$ whose tangent space at $\lambda_\mathbb{R}$ is $T_{\lambda_\mathbb{R}} \mathcal{O} \simeq \mathfrak{g}_0/\mathfrak{g}(\lambda_\mathbb{R})_0$. Then there is a natural non-degenerate skew-symmetric two form $\omega_{\lambda_\mathbb{R}}$ given by $$\omega_{\lambda_\mathbb{R}}(X + \mathfrak{g}(\lambda_\mathbb{R})_0, Y + \mathfrak{g}(\lambda_\mathbb{R})_0) = \lambda_\mathbb{R}\left([X,Y]\right) \qquad (X,Y \in \mathfrak{g}_0).$$ Note that the action of $G$ on the orbit $\mathcal{O}$ induces an action of $G(\lambda_\mathbb{R})$ on its tangent space $T_{\lambda_\mathbb{R}} \mathcal{O}$ at $\lambda_\mathbb{R}$, for which $\omega_{\lambda_\mathbb{R}}$ is preserved under this action. Thus we obtain a symplectic representation $\omega_{\lambda_\mathbb{R}}$ of $G(\lambda_\mathbb{R})$ $$G(\lambda_\mathbb{R}) \longrightarrow Sp(T_{\lambda_\mathbb{R}}\mathcal{O}, \omega_{\lambda_\mathbb{R}}) = Sp(\omega_{\lambda_\mathbb{R}}).$$
Observe that the real symplectic group $Sp(\omega_{\lambda_\mathbb{R}})$ has a two-fold covering group $Mp(\omega_{\lambda_\mathbb{R}})$, called the metaplectic group. The pull-back of this group over $G(\lambda_\mathbb{R})$ defines a two-fold covering group $$\widetilde{G(\lambda_\mathbb{R})} = Mp(\omega_{\lambda_\mathbb{R}}) \times_{Sp(\omega_{\lambda_\mathbb{R}})} G(\lambda_\mathbb{R})$$ of $G(\lambda_\mathbb{R})$, called the metaplectic cover of $G(\lambda_\mathbb{R})$, fitting into the exact sequence $$1 \rightarrow \{1,\epsilon\} \rightarrow \widetilde{G(\lambda_\mathbb{R})} \rightarrow G(\lambda_\mathbb{R}) \rightarrow 1.$$
We are not going to recall the general notion of admissibility here, instead as observed in \cite{E8}, for nilpotent coadjoint orbits of real reductive groups, we have the following simplified form of admissibility, which we will be taken as definition here.

\begin{definition}
Let $G$ be a real reductive group of Harish-Chandra's class. Suppose that $\lambda_\mathbb{R} \in \mathcal{N}^*_\mathbb{R}$ and $\widetilde{G(\lambda_\mathbb{R})}$ are as defined above. Then
a representation $\chi$ of $\widetilde{G(\lambda_\mathbb{R})}$ is admissible if $\chi(\epsilon)= -I$, and $\chi|_{\widetilde{G(\lambda_\mathbb{R})}_0}$ is trivial, where $\widetilde{G(\lambda_\mathbb{R})}_0$ denotes the identity component of $\widetilde{G(\lambda_\mathbb{R})}$. A pair $(\lambda_\mathbb{R}, \chi)$ consisting of a nilpotent element $\lambda_\mathbb{R} \in \mathcal{N}^*_\mathbb{R}$ and an irreducible unitary admissible representation $\chi$ of $\widetilde{G(\lambda_\mathbb{R})}$ is called a nilpotent admissible orbit datum for $G$. In this case, the nilpotent orbit of $\lambda_\mathbb{R}$ is said to be $\mathbb{R}$-admissible.
\end{definition}

When $G$ happens to be complex reductive, we have further simplification:

\begin{lemma} \label{lemma:2.3}
If $G$ is complex reductive, then the metaplectic cover $\widetilde{G(\lambda_\mathbb{R})}$ is always trivial, i.e. $$\widetilde{G(\lambda_\mathbb{R})} \simeq \{1,\epsilon\} \times G(\lambda_\mathbb{R}).$$
\end{lemma}
\begin{proof}
It follows from the fact that if $G$ is complex reductive, then there is a complex symplectic form $\Omega$ on the orbit $\mathcal{O}$ so that the real symplectic representation $G(\lambda_\mathbb{R}) \rightarrow Sp(\omega_{\lambda_\mathbb{R}})$ factors through the complex symplectic group $Sp(\Omega_{\lambda_\mathbb{R}})$, which is simply connected, i.e. $$G(\lambda_\mathbb{R}) \rightarrow Sp(\Omega_{\lambda_\mathbb{R}}) \hookrightarrow Sp(\omega_{\lambda_\mathbb{R}}).$$
As $Sp(\Omega_{\lambda_\mathbb{R}})$ is simply connected, $\widetilde{Sp(\Omega_{\lambda_\mathbb{R}})} = \{1,\epsilon\} \times Sp(\Omega_{\lambda_\mathbb{R}})$ and so
\begin{align*}
\widetilde{G(\lambda_\mathbb{R})} &= Mp(\omega_{\lambda_\mathbb{R}}) \times_{Sp(\omega_{\lambda_\mathbb{R}})} G(\lambda_\mathbb{R})\\
&\simeq \widetilde{Sp(\Omega_{\lambda_\mathbb{R}})} \times_{Sp(\Omega_{\lambda_\mathbb{R}})} G(\lambda_\mathbb{R})\\
&\simeq \{1,\epsilon\} \times G(\lambda_\mathbb{R})
\end{align*}
\end{proof}

\begin{proposition} \label{prop:2.3}
If $G$ is complex reductive, then the set of admissible data $(\lambda_\mathbb{R}, \chi)$ is in one-to-one correspondence with the set of irreducible representations of the component group $G(\lambda_\mathbb{R})/ G(\lambda_\mathbb{R})_0$ of the orbit $G \cdot \lambda_\mathbb{R}$.
\end{proposition}
\begin{proof}
From Lemma \ref{lemma:2.3}, we see that $\widetilde{G(\lambda_\mathbb{R})}$ is trivial. Note that there is an obvious irreducible unitary admissible representation, namely $\sigma \otimes 1$, where $\sigma$ is the sign representation of $\{1,\epsilon\}$. Then for any other irreducible admissible representation, tensoring with $\sigma \otimes 1$ yields a representation which descends to $G(\lambda_\mathbb{R})$ and is trivial on $G(\lambda_\mathbb{R})_0$, thus well-define an irreducible representation of the component group. Clearly such a construction can be reversed easily by lifting the representation from the component group to that of $\widetilde{G(\lambda_\mathbb{R})}$ and then tensoring with $\sigma \otimes 1$.
\end{proof}

\begin{remark}
The same conclusion as in Proposition \ref{prop:2.3} is also valid for more general $G$ as long as one has a one-dimensional unitary admissible representation of $\widetilde{G(\lambda_\mathbb{R})}$. The proof goes in exactly the same line.
\end{remark}

Now we establish a corresponding notion of admissibility for nilpotent $K_\mathbb{C}$-orbits in $\left(\mathfrak{g}/\mathfrak{k}\right)^*$. Let $\lambda_\theta \in \mathcal{N}^*_\theta$ and $\mathcal{O}_\theta$ be the corresponding orbit. Similar to the above situation, the isotropy group $K_\mathbb{C}(\lambda_\theta)$ at $\lambda_\theta$ acts on the tangent space $T_{\lambda_\theta} \mathcal{O}_\theta \simeq \mathfrak{k}_\mathbb{C}/\mathfrak{k}_\mathbb{C}(\lambda_\theta)$ at $\lambda_\theta$ and thus there is an induced action of $K_\mathbb{C}(\lambda_\theta)$ on the top exterior power of the cotangent space, which gives rise to an algebraic character $\gamma: K_\mathbb{C}(\lambda_\theta) \rightarrow \mathbb{C}^*$ defined by $$\gamma(k) = \det \left( \left. Ad^*(k)\right|_{\left(\mathfrak{k}_\mathbb{C}/ \mathfrak{k}_\mathbb{C}(\lambda_\theta)\right)^*} \right) \qquad \left(k \in K_\mathbb{C}(\lambda_\theta)\right).$$

\begin{definition}
With the above notations, a representation $\chi$ of $K_\mathbb{C}(\lambda_\theta)$ is said to be $\theta$-admissible if $d\chi = \frac{1}{2} d \gamma(\lambda_\theta)\cdot I.$ A pair $(\lambda_\theta, \chi)$ consisting of a nilpotent element $\lambda_\theta \in \mathcal{N}^*_\theta$ and an irreducible unitary admissible representation $\chi$ of $K_\mathbb{C}(\lambda_\theta)$ is called a nilpotent $\theta$-admissible datum for $K_\mathbb{C}$. In this case, the nilpotent orbit is called $\theta$-admissible.
\end{definition}

The admissible data of $G$ corresponds nicely to the $\theta$-admissible data of $K_\mathbb{C}$ through the Vergne's version of the Sekiguchi correspondence, we summerize this result in the foolowing theorem.

\begin{theorem} \label{thm:2.6}
Suppose that $G$ is a real reductive Lie group. Adopting to the above notations, there is a natural one-to-one correspodence between the $G$-orbits on $\mathcal{N}_\mathbb{R}^*$ and the $K_\mathbb{C}$-orbits on $\mathcal{N}_\theta^*$ for which each corresponding pairs of orbits are $K$-equivariantly diffeomorphic. Moreover, if $\lambda_\mathbb{R}\in \mathcal{N}_\mathbb{R}^*$ and $\lambda_\theta \in \mathcal{N}_\theta^*$ are the corresponding elements under the diffeomorphism, there is a natural bijection from the admissible representations of $\widetilde{G(\lambda_\mathbb{R})}$ to the $\theta$-admissible representations of $K_\mathbb{C}(\lambda_\theta)$.
\end{theorem}

With Proposition \ref{prop:2.3} in mind, we would further simplify Conjecture \ref{conj:1.1} into complex reductive group case.

Note that $G$ is complex means that there is a natural isomorphism $\psi$ between $G$ and the complexification $K_{\mathbb{C}}$ of $K$. As the action of $K$ on $(\mathfrak{g}/ \mathfrak{k})^*$ respects the complex structure, it extends naturally and uniquely to a holomorphic action of $K_{\mathbb{C}}$. The isomorphism $\psi$ takes the action of $G$ to the action of $K_\mathbb{C}$. In particular, $\psi : \mathcal{N}^*_\mathbb{R} \cong \mathcal{N}^*_{\theta}$. This is a $K$-equivariant diffeomorphism carrying $G$ orbit to $K_{\mathbb{C}}$ orbits. By Lemma \ref{lemma:2.3}, the metaplectic covers were all trivial in this case, so that $\mathbb{R}$-admissible orbit data were identified with $G$-equivariant local systems on orbits. Similarly, all $K_{\mathbb{C}}$ orbits are symplectic, so that characters $\gamma(k)$ are all trivial, and $\theta$-admissible orbit data are identified with $K_{\mathbb{C}}$-equivariant local systems on orbits. And since all orbits have even complex dimensions, the codimension conditions in Conjecture \ref{conj:1.1} are automatic. Thus we obtain Conjecture \ref{conj:1.1} in the form of Conjecture \ref{conj:1.2} in complex reductive group case. 

In our case, the nilpotent orbit $\mathcal{O}$ of dimension 8 in $G_2$, it is shown that the fundamental group $\pi_1(\mathcal{O})$ is trivial. As $\pi_1(\mathcal{O})$ is isomorphic to $G(\lambda) / G(\lambda)_0$, $\chi$ is trivial and $V_{\chi}= \mathbb{C}$. Thus, the space of algebraic sections of the bundle $\mathcal{V}_{\chi} = K_\mathbb{C} \times_{K_{\mathbb{C}} (\lambda)} \mathbb{C}$ is just the ring of regular functions, $R(\mathcal{O})$,  on the orbit $\mathcal{O}$. Hence, we can reduce to calculate $R(\mathcal{O})$.

\subsection{Calculation of rings of functions on nilpotent orbits} \label{sec:2}
To calculate the $K$-multiplicities in representations attached to nilpotent orbits, we will use McGovern's method. McGovern works with nilpotent elements of $\mathfrak{g}$ instead of nilpotent elements of $\mathfrak{g}^*$. An invariant bilinear form on $\mathfrak{g}$ can identify $\mathfrak{g}^*$ and $\mathfrak{g}$. 

First, define
$$\mathcal{N}_{\theta} = \{ e \in \mathfrak{p} | ad(e) \text{ is nilpotent } \}.$$
An Ad$(G)$-invariant, $\theta$-invariant bilinear form on $\mathfrak{g}_0$, positive definite on $\mathfrak{p}_0$ and negative definite on $\mathfrak{k}_0$, provides an isomorphism $\mathfrak{g} \cong \mathfrak{g}^*$ that carries $\mathcal{N}_{\theta}$ $K_{\mathbb{C}}$ equivariantly onto $\mathcal{N}^*_{\theta}$. Then fix $e \in \mathcal{N}_{\theta} \subset \mathfrak{p}$. Using the Jacobson-Morozov theorem, there exists $f \in \mathcal{N}_{\theta} \subset \mathfrak{p}$ and a semisimple element $h \in \mathfrak{k}$ such that
$$[h,e] = 2e, \text{  } [h,f] =-2f, \text{  } [e,f]=h.$$
By the representation theory of $\mathfrak{sl}(2)$, $ad(h)$ has integral eigenvalues. Thus we may define for $m \in \mathbb{Z}$
$$\mathfrak{g}(m) = \{ z \in \mathfrak{g} | [h,z] = mz \},  \mathfrak{k}(m) = \mathfrak{g} (m) \cap \mathfrak{k}, \mathfrak{p}(m) = \mathfrak{g} (m) \cap \mathfrak{p}, $$
These spaces define gradings of $\mathfrak{g}$, $\mathfrak{k}$,
$\mathfrak{p}$. Define the parabolic subalgebra $\mathfrak{q}$ by
$$\mathfrak{q} = \sum_{m \geq 0} \mathfrak{k}(m), \mathfrak{u} = \sum_{m >0} \mathfrak{k}(m), \mathfrak{l} = \mathfrak{k}(0). $$
Here we define
$$L_{\mathbb{C}} = \{k \in K_{\mathbb{C}} | Ad(k)h =h \}, $$
and let $U$ be the connected subgroup with Lie algebra $\mathfrak{u}$. Write $Q= L_{\mathbb{C}} U$ for the corresponding parabolic subgroup of $K_{\mathbb{C}}$. 
Finally, we define
$$\mathfrak{o} = \sum_{m \geq 2} \mathfrak{p}(m).$$
Instead of studying the action of $K_{\mathbb{C}}$ on $K_{\mathbb{C}} \cdot e$ directly, McGovern first studies $Q \cdot e $.

Note that $Q$ acts algebraically on the smooth variety $\mathfrak{o}$, we can form the fiber product $Z= K_{\mathbb{C}} \times_Q \mathfrak{o} = (K_{\mathbb{C}} \times \mathfrak{o}) /\sim$ , where $(x,z) \sim (x',z') $ if $x = x'q $ and $z' = $ Ad$(q)z$ for some $q\in Q$. 

This space $Z$ is a smooth variety with an action of $K_\mathbb{C}$; it is the total space of a homogeneous vector bundle on the flag variety $K_\mathbb{C} / Q$. Thus
$$(x,z ) \sim (x', z') \text{ implies Ad} (x) z = \text{ Ad} (x') z'.$$
Therefore we have  an algebraic map $\pi : Z \rightarrow \mathcal{N} _{\theta}$, defined by $(x,z) \mapsto $ Ad$(x)z$.

With these setting in mind, we can now state a theorem which tells us why it is reasonable to consider $Q$ first.

\begin{theorem}
The map $\pi$ is proper and birational, with image equal to the closure of $K_\mathbb{C} \cdot e$. Consequently the algebra $R(Z)$ of regular functions on $Z$ is natually isomorphic to the normalization of $R(\overline{K_{\mathbb{C}} \cdot e})$.
Assume in addition that the boundary of the orbit closure $\overline{K_{\mathbb{C}} \cdot e}$ has complex codimension at least two. Then $R(Z)$ is naturally isomorphic to the algebra $R(K_\mathbb{C} \cdot e)$ of regular functions on the orbit.
\end{theorem} 
Thus we can now turn into the study of $R(Z)$. Luckily, we have the following decomposition at hand.

\begin{theorem} \label{thm:3.7}
$R(Z) \cong Ind_Q^{K_\mathbb{C}} (R(\mathfrak{o}) ) \cong \sum_{k = 0}^{\infty} Ind_Q^{K_\mathbb{C}} (S^k(\mathfrak{o}^*) ).$
\end{theorem}

The above formula directs us to understand the induction functor from $Q$ to $K_{\mathbb{C}}$ and $S^k(\mathfrak{o}^*)$ as a representation of $Q$. First we would need the following settings.

Suppose $G$ is a connected simple complex Lie group, $Q \subset G $ is a parabolic subgroup, and $H \subset Q$ is a maximal torus. Let $U$ be the unipotent radical of $Q$ and $L$ for the Levi subgroups containing $H$. Choose a system of positive roots $\Delta^+(\mathfrak{l},\mathfrak{h}) $ for $\mathfrak{h}$ in $\mathfrak{l}$, so that we can extend it to positive roots system for $H$ in $G$. Explicitly, $\Delta^+(\mathfrak{g}, \mathfrak{h}) = \Delta^+(\mathfrak{l}, \mathfrak{h}) \cup \Delta^+(\mathfrak{u}, \mathfrak{h})$ is a positive system for $\mathfrak{h}$ in $\mathfrak{g}$. Write $X^*(H) \supset \Delta (\mathfrak{g}, \mathfrak{h})$ for the lattice of weights; $X_*(H) \supset \check{\Delta} (\mathfrak{g}, \mathfrak{h})$ for the dual lattice of coweights. If we fix an identification of the Lie algebra of $\mathbb{C}^{\times}$ with $\mathbb{C}$, then
$$X^* (H) \subset \mathfrak{h}^*, X_*(H) \subset \mathfrak{h}, $$
and these inclusions are compatible with the natural dualities
$$X^*(H) \times  X_*(H) \rightarrow \mathbb{Z}, \text{  } \mathfrak{h}^* \times \mathfrak{h} \rightarrow \mathbb{C}.$$
If we identify coweights with elements of $\mathfrak{h}$ as above, then the coroot $\check{\alpha}$ corresponding to a root $\alpha$ is the element called $h_{\alpha}$. 
A weight $\lambda \in X^*(H) $ is defined as $G$-dominant if for every positive root $\alpha \in \Delta(\mathfrak{g}, \mathfrak{h})$, $\lambda({\check{\alpha}}) \geq 0$.

Then apply the Bott-Borel-Weil Theorem on Theorem \ref{thm:3.7}, we have the following refornulation.
\begin{theorem}
Suppose $\lambda \in X^*(H)$ is $G$-dominant, $V_\lambda$ is the irreducible algebraic representation of $G$ of highest weight $\lambda$. Let $\rho$ be the half sum of positive roots,  $E$ be a finite dimensional algebraic representation of $Q$, $E_\mu$ be the irreducible representation of $L$ of highest weight $\mu$. Then the multiplicity of $V_\lambda^*$ in the virtual representation $\sum_p (-1)^p (Ind_Q^G)^p (E^*)$ is equal to 
$$\sum_w sgn(w) \cdot  (\text{multiplicity of }E_{w(\lambda + \rho_c) - \rho_c} \text{ in } E).$$
where $w\in W(\mathfrak{k}, \mathfrak{h})$ is such that $w(\lambda + \rho_c) - \rho_c$ is $L$-dominant, the $E_{w(\lambda + \rho_c) - \rho_c}$ is the irreducible representations of $L$ with highest weight $w(\lambda + \rho_c) - \rho_c$. 
\end{theorem}

Even better, we have the following vanishing theorem.
\begin{theorem} \label{thm:3.12}
If $G$ is complex, then $ (Ind_Q^G)^p (S(\mathfrak{o})^*) = 0 $ for all $p >0$.
\end{theorem}

Combining these results, we are led to the following useful theorem for us to calculate the ring of regular functions on orbits.
\begin{proposition} \label{prop:1.1}
Suppose $G$ is complex. Fix a maximal torus $H_{\mathbb{C}} \subset L_{\mathbb{C}}$ and a system of positive roots $\Delta^+ (\mathfrak{l}, \mathfrak{h})$. Extend this to a system of positive roots for $H_{\mathbb{C}} $ in $K_{\mathbb{C}}$ as above, and denote $\rho_c$ as the half sum of the positive roots. Fix a $K_{\mathbb{C}}$-dominant weight $\lambda \in X^*(H_{\mathbb{C}})$, and write $V_{\lambda}$ for the corresponding irreducible representation of $K_{\mathbb{C}}$. Then the multiplicity of $V_{\lambda}^*$ in the ring of regular functions on the normalization of $\overline{K_{\mathbb{C}} \cdot e}$ is equal to
$$\sum_w sgn(w) \cdot  (\text{multiplicity of }E_{w(\lambda + \rho_c) - \rho_c} \text{ in } S(\mathfrak{o})),$$
where $w \in W(\mathfrak{k}, \mathfrak{h})$ such that $w(\lambda + \rho_c) -\rho_c$ is $L_{\mathbb{C}}$ - dominant.
\end{proposition}
\section{The Model Orbit in $G_2(\mathbb{C})$} 

\subsection{The ring of regular functions} \label{sec:3.1}
Let $G$ be a complex simply connected simple Lie group of type $G_2$ with $\mathfrak{g}_0$ being its Lie algebra. Fix a Cartan subgroup $H$, thus the corresponding Cartan subalgebra $\mathfrak{h}_0$ of $\mathfrak{g}_0$ and a system of positive roots $\Phi ^+$. The corresponding base is given by $\Delta = \{ \alpha, \beta \}$ with $\alpha$ being the short root and $\beta$ being the long root. Let $h_{\alpha}$, $h_{\beta}$ be the coroot of $\alpha$, $\beta$ respectively.

Let $\mathcal{O}$ be the unique nilpotent orbit of complex dimension 8 in $\mathfrak{g}_0$. By the Jacobson-Morozov Theorem, any representative $e$ of $\mathcal{O}$ embeds into a standard triple $\{h, e, f\}$. By conjugation, we can assume $h \in \mathfrak{h}_0$ is $\Delta$-dominant and integral. This choice of $h$ is the one used to define the weighted Dynkin diagram of the orbit and it is known that the corresponding weights are $\alpha (h) = 1$, $\beta (h) =0$. In other words, we have $h = 2 h_{\alpha} + 3 h_{\beta}$. Define
$$L = \{ g \in G \,|\, \text{Ad}(g)h =h\}.$$
Let $\mathfrak{l}_0$ be the Lie algebra of $L$. According to the properties of $h$, we have 
$$\mathfrak{l}_0 = \mathfrak{h}_0 \oplus \mathfrak{g}_{\beta} \oplus \mathfrak{g}_{-\beta}.$$
More precisely $\mathfrak{g}_0$ can be decomposed into
$\mathfrak{g}_0 = \bigoplus_{i=-3}^{3} \mathfrak{g}_0 (i)$ with
respect to the action of ad\,$h$, where
\begin{align*}
\mathfrak{g}_0 (0) &= \mathfrak{l}_0 \\
\mathfrak{g}_0 (1) &= \mathfrak{g}_{\alpha + \beta} \oplus \mathfrak{g}_{\alpha}\\
\mathfrak{g}_0 (2) &= \mathfrak{g}_{2\alpha + \beta} \\
\mathfrak{g}_0 (3) &= \mathfrak{g}_{3\alpha + 2\beta} \oplus \mathfrak{g}_{3 \alpha + \beta}\\
\mathfrak{g}_0 (-1) &= \mathfrak{g}_{-\alpha - \beta} \oplus \mathfrak{g}_{-\alpha}\\
\mathfrak{g}_0 (-2) &= \mathfrak{g}_{-2\alpha - \beta} \\
\mathfrak{g}_0 (-3) &= \mathfrak{g}_{-3\alpha - 2\beta} \oplus
\mathfrak{g}_{-3 \alpha - \beta}.
\end{align*}
Define
$$\mathfrak{o}_0 = \mathfrak{g}_0 (2) \oplus \mathfrak{g}_0(3)= \mathfrak{g}_{2\alpha + \beta} \oplus \mathfrak{g}_{3\alpha + 2\beta} \oplus \mathfrak{g}_{3 \alpha + \beta}.$$

Let $X^*(H)$ be the weight lattice and fix a $G$-dominant weight $\lambda \in X^*(H)$. Write $V_{\lambda}$ for the corresponding irreducible representation of $G$ with highest weight $\lambda$. According to Proposition \ref{prop:1.1}, the multiplicity of $V_{\lambda}^*$ in the ring of regular functions on $G \cdot e $ is equal to
$$\sum sgn(w) \cdot (\text{ multiplicity of }E_{w(\lambda +\rho) -\rho} \text{ in } S(\mathfrak{o}_0)),$$
where the sum is over $w \in W$ such that $w(\lambda + \rho) - \rho$ is $L$-dominant, and $E_{w(\lambda +\rho) -\rho} $ is the irreducible algebraic representation of $L$ of highest weight $w(\lambda +\rho) -\rho$. 

Thus, it suffices to calculate $S(\mathfrak{o}_0)$ as a representation of $L$.

\subsection{Decomposition of $S(\mathfrak{o}_0)$}
In this section, we will work in the setting in Section \ref{sec:3.1}, and calculate $S(\mathfrak{o}_0)$ as a representation of $L$. Note that $\mathfrak{l}_0 \cong \mathfrak{gl}(2)$ and we fix an isomorphism 
$$j: \mathfrak{l}_0 \rightarrow \mathfrak{gl} (2)$$
carrying $h_{\beta}, e_{\beta}, f_{\beta}$ to the standard triple in $\mathfrak{sl}(2) \subset \mathfrak{gl}(2)$, and $h$ to the identity matrix. 
Via this identification, we will denote a weight $\lambda \in \mathfrak{h}_0^*$ by the pair $(\lambda(h),\lambda(h_\beta))$ so that the first value gives the action of the identity and the second shows the weight of the standard $\mathfrak{sl}(2)$ subalgebra.

Write $(\tau, \mathbb{C}^2_{\tau})$ for the standard representation of $\mathfrak{l}_0$; its weights are
$ (1,1), (1,-1) .$
The one-dimensional determinant character $(\det, \mathbb{C}_{\det})$ has weight
$ (2,0).$
Thus, $\tau$ twisted by $\det$ has weights
$ (3,1), (3,-1).$

As an algebraic representation of a reductive group is determined up to equivalence by its set of weights with multiplicities, we have
$$\mathfrak{o}_0 = \mathbb{C}_{\det} \oplus (\mathbb{C}_{\det} \otimes \mathbb{C}^2_{\tau} ).$$
Since $\mathfrak{sl}(2)$ acts trivially on $\mathbb{C}_{\det^k}$ for all $k$, $S^q(\mathbb{C}_{\det} \otimes \mathbb{C}^2_\tau)$ is isomorphic to $S^q(\mathbb{C}^2_\tau)$ as $\mathfrak{sl}(2)$-representations. Therefore $S^q(\mathbb{C}_{\det} \otimes \mathbb{C}^2_\tau)$ must be of the form $\mathbb{C}_{\det^k} \otimes S^q(\mathbb{C}^2_\tau)$ for some $k \in \mathbb{N}$.

Note that $h$ lies in the center of $\mathfrak{l}_0$, it must act as scalar multiplication on each representation. In particular, we see that $h$ acts on $\mathbb{C}_{\det} \otimes \mathbb{C}^2_\tau$ as multiplication by 3. More generally for any indecomposable vector $(v \otimes w_1) \cdots (v \otimes w_q) \in S^q (\mathbb{C}_{\det} \otimes \mathbb{C}^2_{\tau})$, $v \in \mathbb{C}_{\det}, w_i \in \mathbb{C}^2_\tau$, we have
\begin{align*}
&h\left((v\otimes w_1) \cdots (v \otimes w_q)\right)\\
= & \sum_i (v\otimes w_1) \cdots h(v \otimes w_i)\cdots (v \otimes w_q)\\
= & \sum_i (v\otimes w_1) \cdots (hv \otimes w_i + v \otimes h w_i)\cdots (v \otimes w_q)\\
= & 3 \sum_i (v\otimes w_1) \cdots  (v \otimes w_q) \\
= & 3q (v\otimes w_1) \cdots (v \otimes w_q).
\end{align*}


Similarly, for $v \otimes (w_1 \cdots w_q) \in \mathbb{C}_{\det^k}
\otimes S^q (\mathbb{C}^2_{\tau})$, we have
\begin{align*}
& h(v \otimes (w_1 \cdots w_q))\\
=& hv \otimes (w_1 \cdots w_q) + v \otimes h(w_1 \cdots w_q)\\
=& 2kv \otimes (w_1 \cdots w_q)+ v \otimes \sum_i\left(w_1 \cdots (hw_i) \cdots w_q\right)\\
=& 2k \otimes (w_1 \cdots w_q) + q v \otimes (w_1 \cdots w_q)\\
=& (2k+q) v \otimes (w_1 \cdots w_q)
\end{align*}
It forces that $3q = 2k+q$, i.e. $k = q$.

Hence
\begin{align*}
S^k (\mathfrak{o}_0) &= S^k (\mathbb{C}_{\det} \oplus \mathbb{C}_{\det} \otimes \mathbb{C}^2_{\tau} ) \\
&= \bigoplus_{p+q = k} S^p\mathbb{C}_{\det} \otimes S^q \left(\mathbb{C}_{\det} \otimes \mathbb{C}^2_{\tau} \right) \\
&= \bigoplus_{p+q = k} \mathbb{C}_{\det^p} \otimes \mathbb{C}_{\det^q } \otimes S^q (\mathbb{C}^2_{\tau}) \\
&= \bigoplus_{p+q = k} \mathbb{C}_{\det^k} \otimes S^q (\mathbb{C}^2_{\tau})\\
&= \bigoplus_{q=0}^k \mathbb{C}_{\det^k} \otimes S^q
(\mathbb{C}^2_{\tau}).
\end{align*}

Finally, we get
$$S(\mathfrak{o}_0) \cong \bigoplus_{k \in \mathbb{N}} \bigoplus_{q=0}^k \mathbb{C}_{\det^k} \otimes S^q (\mathbb{C}^2_{\tau}).$$

It remains to calculate the multiplicities of some irreducible $L$-representations in $S(\mathfrak{o}_0)$. 

Suppose $\lambda= a \alpha + b \beta$ is the highest weight of
$\mathbb{C}_{\det^k} \otimes S^q (\mathbb{C}^2_{\tau})$ for
$\mathfrak{l}_0$. Then
$$\begin{array}{rll}
a &= \lambda (h) &= 2k +q\\
-a +2b &= \lambda (h _{\beta}) &=q.
\end{array}$$
If follows that $a = 2k+q$, $b = k+q$ and $\lambda = (2k +
q)\alpha + (q+k)\beta$.

Define $S(\lambda)$ to be the set of $w \in W$ for which $w(\lambda + \rho) - \rho $ is the highest weight of an irreducible $L$-representation appears in $S(\mathfrak{o}_0)$. From the above calculations, we see that
$$S(\lambda) = \{ w \in W | w(\lambda +\rho) - \rho = (2k + q)\alpha + (k+q)\beta, \text{for some integers $k \geq q \geq 0$} \}.$$

\begin{proposition} \label{thm:3.1}
$S(\lambda) = \{ 1\}$ for all $G$-dominant weight $\lambda$.
\end{proposition}
\begin{proof}
For $w\in S(\lambda)$, $$w(\lambda +\rho) - \rho = (2k + q)\alpha
+ (k+q)\beta$$ for some integers $k \geq q \geq 0$. Then we have
\begin{align*}
(w(\lambda + \rho) - \rho)(h_{\beta} )= q\\
(w(\lambda + \rho) - \rho)(h) = 2k+q.
\end{align*}
The condition $q \geq 0$ is equivalent to
$$(w(\lambda + \rho)) (h_\beta) \geq 1,$$
which is also equivalent to $w^{-1} (\beta) >0$ since $\lambda +
\rho$ is always a strictly dominant weight. On the other hand,
notice that
\begin{align*}
& q \leq k\\
\Leftrightarrow \quad & 2q \leq \left(w(\lambda + \rho) - \rho\right) (h) - q \\
\Leftrightarrow \quad & (w(\lambda + \rho))(h_{\alpha}) \geq 1\\
\Leftrightarrow \quad & w^{-1} (\alpha) >0.
\end{align*}
Conversely, if $w \in W$ satisfies $w^{-1}(\alpha)>0$ and
$w^{-1}(\beta)>0$, define
\begin{align*}
q = \left(w(\lambda+\rho) -\rho \right) (h_\beta),\\
k = \left(w(\lambda+\rho) -\rho \right) (h_\alpha + h_\beta).
\end{align*}
As $w^{-1}(\alpha)>0, w^{-1}(\beta)>0$, clearly $q,k$ are positive
integers satisfying $k \geq q$. Moreover,
\begin{align*}
w(\lambda+\rho) -\rho &= \left(w(\lambda+\rho) -\rho \right)
(h_\alpha) w_\alpha + \left(w(\lambda+\rho) -\rho \right)
(h_\beta) w_\beta\\
&= (k-q) w_\alpha + q w_\beta\\
&= (k-q) (2 \alpha + \beta) + q (3\alpha + 2 \beta)\\
&= (2k + q)\alpha + (k+q)\beta
\end{align*}
It follows that $w\in S(\lambda)$. Finally, since the Weyl group
$W$ acts simply transitively on the set of bases, we have
$S(\lambda)=\{1\}$.
\end{proof}

Combining the results in Proposition \ref{prop:1.1} and Proposition \ref{thm:3.1}, we obtain

\begin{proposition} \label{prop:2}
Suppose $G$ is a complex algebraic group of type $G_2$ and $\mathcal{O}$ is the nilpotent orbit of complex dimension 8. The the algebra $R(\mathcal{O})$ of regular functions on $\mathcal{O}$ contains every irreducible (algebraic) representation of $G$ with multiplicity one.
\end{proposition}

\begin{proof}
Observe that each summand $\mathbb{C}_{det^k}  \otimes  S^q (\mathbb{C}^2_\tau) $ in $S(\mathfrak{o}_0)$ has different highest weight. By Proposition \ref{thm:3.1}, the multiplicity of $V_{\lambda}$ in the ring of regular functions on $\mathcal{O}$ is equal to 
$$ \sum_{w \in S(\lambda)}  sgn(w)= sgn(1) = 1. $$
\end{proof}
\subsection{Quantization of the Model Orbit}  \label{sec:last}

The next step is to find out a unitary representation $\pi (\mathcal{O})$ attached to this nilpotent orbit $\mathcal{O}$. As proposed by McGovern in \cite{McG}, $\pi( \mathcal{O}) $ should be the irreducible spherical representation of $G$ with parameter one fourth the sum of positive roots. We will show that this irreducible spherical representation is unitary based on two other results. 

Instead of attaching the nilpotent orbit $\mathcal{O}$ to a representation $\pi$ directly, it would be easier to state the reverse direction first. 
Given a simple Lie algebra $\mathfrak{g}$. We would first replace $\pi$ by its annihilator Ann$(\pi)$, called primitive ideal, in the universal enveloping algebra $U(\mathfrak{g})$ of $\mathfrak{g}$. With a primitive ideal $I \subset U(\mathfrak{g})$ in mind, identify the graded algebra gr$U(\mathfrak{g})$ with the symmetric algebra $S(\mathfrak{g})$ via the Poincare-Birkhoff-Witt Theorem. Then the associated variety $\mathcal{V}(I)$  is defined to be the zeros of the associated graded ideal gr$I \subset S(\mathfrak{g})$ in $\mathfrak{g}^*$. A theorem by Borho, Brylinski and Joseph states that $ \mathcal{V}(I)$ is the closure $\overline{\mathcal{O}_I}$ of a unique nilpotent orbit $\mathcal{O}_I$. In this way, we say that $\mathcal{O}_I$ is associated to $I$. 

Therefore, we would like to have a closer look of the correspondence between nilpotent orbits and primitive ideals. Now consider a complex simple Lie algebra $\mathfrak{g}$ which is not of type $A_n$. Let $\mathcal{O}_{min}$ denote the unique minimal non-zero nilpotent orbit. Joseph has shown that there is a unique completely prime primitive ideal, Joseph ideal $J_0$, associated to $\mathcal{O}_{min}$. On the other hand, in our case, Joseph has shown that there are exactly two primitive ideals associated to the model orbit $\mathcal{O}$ in $G_2$.

As we would need to bring in the group $SO(7)$ later, let us denote $\mathfrak{g}_2$ as the Lie algebra of the complex Lie group $G$ of type $G_2$ in this section. Let $\alpha$, $\beta$ be simple roots for $\mathfrak{g}_2$ with $\alpha$ short and $\beta $ long; $\omega_1$, $\omega_2$ be the fundamental weights corresponding to $\alpha$ and $\beta$; $\rho $ be the half sum of positive roots. Then for $\lambda \in \mathfrak{h}^*$, let $M(\lambda)$ be the Verma module of highest weight $\lambda - \rho$, $L(\lambda)$ be the unique simple factor module of $M(\lambda)$ and set $J(\lambda) = $ Ann$ L(\lambda)$, the annihilator of $L(\lambda)$. The primitive ideals assicoated to $\mathcal{O}$ would then be $J_1 = J( \frac{1}{2} (\omega_1 + \omega_2))$ and $J_2 = J( \frac{1}{2} (5\omega_1 - \omega_2))$.

Then consider $G$ as the connected algebraic subgroup of $SO(7)$. Let $\mathfrak{so}(7)$ be the Lie algebra of $SO(7)$, $J_0$ be the Joseph ideal associated to the unique minimal orbit in $\mathfrak{so}(7)$. Levasseur and Smith have shown that $J_1= J_0 \cap U(\mathfrak{g}_2)$ from which we can well deinfe an natural embedding from $U(\mathfrak{g}_2)/ J_1$ to $ U(\mathfrak{so}(7))/ J_0$. Furthermore, they have shown the following identification
\begin{theorem} 
\cite{LS} 
The embedding $U(\mathfrak{g}_2)/ J_1 \rightarrow U(\mathfrak{so}(7))/ J_0$ is an equality.
\end{theorem}

To establish the unitarizability of $U(\mathfrak{g}_2)/ J_1$, we can thus first consider the unitarizability of $U(\mathfrak{so}(7))/ J_0$ and then descends the unitary structure to $U(\mathfrak{g}_2)/ J_1$ canonically. Note that $U(\mathfrak{so}(7 ))/ J_0$ is the minimal representation of $\mathfrak{so}(7)$ and Huang has showed that the minimal representation of every simple complex Lie group is unitarizable.

\begin{theorem} \label{thm:huang}
\cite{Huang} Let $G$ be a simple complex Lie group. Then the minimal representation $V_{min} = U(\mathfrak{g})/ J_0$ is unitarizable. 
\end{theorem}

Therefore, $U(\mathfrak{so}(7))/ J_0$ is unitarizable. By Theorem \ref{thm:huang}, $U(\mathfrak{g}_2)/J_1$, which is the restriction of $U(\mathfrak{so}(7 ))/J_0$ as a $G_2$ representation, is also unitarizable. By noting $\frac{1}{2} (\omega_1 + \omega_2) = \frac{1}{2} \rho$, we can use McGovern's result finally.

\begin{theorem} \label{thm:3.5}
\cite{McG} Consider $Q=U(\mathfrak{g_2})/ J(\rho/2)$, the quotient of $U(\mathfrak{g_2})$ by the maximal ideal of infinitesimal character $\rho /2$. Let $G$ act on $U(\mathfrak{g_2})/ J(\rho/2)$ by the adjoint representation. Then it is the sum of all irreducible representations of $K_{\mathbb{C}}$, each occurring with multiplicity one.
\end{theorem} 

From Theorem \ref{thm:3.5}, we observe that the proposed representation $U(\mathfrak{g_2})/ J(\rho/2)$ is unitarizable and its $K$-finite vectors coincide with the ring of regular functions $R(\mathcal{O})$ as $K_\mathbb{C}$-representations. Hence we have established Conjecture \ref{conj:1.2} for the model orbit of complex $G_2$. As suggested also in \cite{E8}, we can use similar approach to tackle the corresponding real case, i.e. Conjecture \ref{conj:1.1} for the nilpotent orbit $\mathcal{O}_\mathbb{R}$ of real dimension $8$ for the noncompact real form $G_{2(2)}$. Although $G_{2(2)}$ is not a complex real form, the vanishing theorem, Theorem \ref{thm:3.12}, is still valid for the this orbit, which follows from a direct application of Bott-Borel-Weil Theorem. Henceforth Proposition \ref{prop:1.1} also works in this case. By similar calculation on the decomposition of $S(\mathfrak{o})$ and the corresponding $L_\mathbb{C}$-dominant conditions on weights, we observe that $\mathcal{O}_\mathbb{R}$ is an "almost model" orbit, in which every irreducible $K_\mathbb{C}$-representation appears with multiplicity at most one. But it is far from being model is the sense that not all irreducible $K_\mathbb{C}$-representations appear at all. Indeed, only those irreducible $K_\mathbb{C}$-representations whose highest weights are contained in the central cone of the $K_\mathbb{C}$-dominant weights generated by $-\alpha-\beta$ and $-\beta$ appear in the ring of regular functions $R(\mathcal{O}_\theta)$, where $\mathcal{O}_\theta$ is the $K_\mathbb{C}$-homogeneous space associated to $\mathcal{O}_\mathbb{R}$ via the Sekiguchi correspondence. 
However, it is unclear which irreducible unitary representation is to be attached to this orbit according to Conjecture \ref{conj:1.1} and that the fundamental group $\pi_1(\mathcal{O}_\theta)$ is nontrivial. For this reason, the details of the corresponding real case are not included here.

\subsection*{Acknowledgements}
The author is very grateful to Jing-song Huang for his guidance for this project.



\bibliographystyle{ieeetr}

\begin{thebibliography}{99}
\bibitem{E8} J. Adams, J. Huang, D. Vogan, \textit{Functions in the model orbit in $E_8$}, Representation Theory, \textbf{232}(1998), pp. 224-263.

\bibitem{book} D.H. Collingwood, W.M. McGovern \textit{Nilpotent Oribts in Semisimple Lie Algebras}, Van Nostrand Reinhold, 1993. 

\bibitem{Huang} J. Huang, \textit{Minimal Representations, Shared Orbits, and Dual Pair Correspondences}, Internat. Math. Res. Notices  \textbf{1995}, No. 6, pp. 309-323. 

\bibitem{LS} T. Levasseur, S.P. Smith, \textit{Primitive Ideals and Nilpotent Orbits in Type $G_2$}, J. Algebra \textbf{114}(1988), pp. 81-105. 

\bibitem{McG} W. McGovern, \textit{Rings of regular functions on nilpotent orbits II: model algebras and orbits}, Comm. Alg. \textbf{22(3)}(1994), pp. 765-772. 

\end{thebibliography}

\end{document}